\documentclass[12pt]{arxivclassfile}

\usepackage{mathrsfs,mathtools}

\usepackage{amsmath,amssymb,amsthm,float}

\usepackage{graphicx,subcaption, verbatim}

\newtheorem{theorem}{Theorem}[section]
\newtheorem{lemma}[theorem]{Lemma}
\newtheorem{proposition}[theorem]{Proposition}

\theoremstyle{definition}

\theoremstyle{remark}

\numberwithin{equation}{section}

\def\R{{\mathbb R}}
\def\dim{{\rm dim}}
\def\N{{\mathbb N}}
\def\dist{{\rm dist}}
\def\:{{\colon}}

\begin{document}

\title{Some results in support of the Kakeya Conjecture}


%    author one information
\author{Jonathan M. Fraser}
\address{School of Mathematics and Statistics, The University of St Andrews, St Andrews,\\ KY16 9SS, UK.}
%\curraddr{}
\ead{jmf32@st-andrews.ac.uk}
%\thanks{}

%    author two information
\author{Eric J. Olson}
\address{Department of
  Mathematics/084, University of Nevada, Reno, NV 89557, USA.}
%\curraddr{Mathematics Institute, University of Warwick, Coventry, CV4 7AL, UK.}
\ead{ejolson@unr.edu}
%\thanks{}

\author{James C. Robinson}
\address{Mathematics Institute, University of Warwick, Coventry, CV4 7AL, UK.}
%\curraddr{}
\ead{j.c.robinson@warwick.ac.uk}
%\thanks{JCR is supported by an EPSRC Leadership Fellowship, Grant EP/G007470/1.}

%    \subjclass is required.
%\subjclass[2010]{Primary }

\date{}

%\dedicatory{}

%    Abstract is required.
\begin{abstract}
A Besicovitch set is a subset of $\R^d$ that contains a unit line segment in every direction and the famous Kakeya conjecture states that Besicovitch sets should have full dimension.  We provide a number of results in support of this conjecture in a variety of contexts.  Our proofs are simple and aim to give an intuitive feel for the problem.  For example, we give a very simple proof that the packing and lower box-counting dimension of any Besicovitch set is at least $(d+1)/2$ (better estimates are available in the literature).  We also study the `generic validity' of the Kakeya conjecture in the setting of Baire Category and prove that typical Besicovitch sets have full upper box-counting dimension.

We also study a weaker version of the Kakeya problem where unit line segments are replaced by half-infinite lines.  We prove that such `half-extended Besicovitch sets' have full Assouad dimension.  This can be viewed as full resolution of a (much weakened) version of the Kakeya problem.

\emph{Mathematics Subject Classification} 2010:  28A80, 54E52, 28A75.

\emph{Key words and phrases}: Kakeya conjecture, box-counting dimension, Hausdorff dimension, Assouad dimension, Baire category.
\end{abstract}

\maketitle

\section{Introduction}

A Besicovitch set in $\R^d$ is a set that contains a unit line segment in every direction.  For convenience (and because most of our results concern box-counting dimension which is usually only defined for bounded sets) we will assume that our Besicovitch sets are bounded. In 1919 Besicovitch proved the surprising result that there are Besicovitch sets that have zero Lebesgue measure, see the later 1928 paper \cite{B}. Given this, it is natural to ask whether Besicovitch sets can be even smaller, i.e.\ whether there are Besicovitch sets with zero $s$-dimensional Hausdorff measure for some $s<d$. The Kakeya Conjecture is that this is not possible: equivalently, that the Hausdorff dimension, $\dim_{\rm H}$,  of any Besicovitch set in $\R^d$ is maximal, i.e.\ equal to $d$.

Davies \cite{D} showed that this conjecture is true when $d=2$. Wolff \cite{W} showed that $\dim_{\rm H}(K)\ge (d+2)/2$, and Katz \& Tao \cite{KT} showed that $\dim_{\rm H}(K)\ge (2-\sqrt2)(d-4)+3$. With the Hausdorff dimension replaced by the upper box-counting dimension, $\dim_{\rm B}$, the best bound in three dimensions is now $\dim_{\rm B}(K)\ge 5/2+\epsilon$; a hard-won improvement over the result of Wolff due to Katz, {\L}aba, \& Tao \cite{XKT}. Slightly better estimates for upper box-counting dimension are also available in higher dimensions, see \cite{KT2} for a comprehensive survey of the state of the art up to around 2002.

In this paper we provide several results in support of the Kakeya conjecture. One of the overarching purposes of the paper is to provide some intuition towards why the Kakeya conjecture should be true, especially to readers not familiar with the `state of the art'.  We present a very simple folklore proof that the lower box or packing dimension of any Besicovitch set is at least $(d+1)/2$. We also show that the box-counting dimension version of the Kakeya Conjecture is true `generically', in the following sense.  Let $\mathring S^{d-1}$ denote the unit sphere in $\R^d$ with antipodal points identified. We encode a representative family of `minimal'  Besicovitch sets in $\R^d$ by bounded maps $f\:\mathring S^{d-1}\to\R^d$, where $f(x)$ gives the centre of the unit line segment oriented in the $x$ direction. Note that any bounded Besicovitch set must contain one of the Besicovitch sets in our family.  Denoting by $B(\mathring S^{d-1})$ the collection of all such maps equipped with the supremum norm, we show that (i) for a dense set of $f$ the corresponding Besicovitch set has positive Lebesgue measure, in contrast to Besicovitch's existence result mentioned above, and (ii) the set of those $f$ for which the corresponding Besicovitch set has maximal upper box-counting (Minkowski) dimension $d$ is a residual subset of $B(\mathring S^{d-1})$.  The Baire Category Theorem has previously been used by K\"orner \cite{K} in the context of the Kakeya problem to prove the existence of zero measure Besicovitch sets.

Motivated by Keleti \cite{keleti} we also study the related question of `Besicovitch sets' which contain half-infinite (or doubly infinite) lines in every direction.  Monotonicity of all the standard notions of dimension mean that these sets have dimension at least as large as the classical Besicovitch sets, thus making the analogous conjecture easier. We solve this problem completely for the `biggest' (and therefore easiest) dimension: Assouad dimension.  This could perhaps be viewed as resolution of  the  `weakest reasonable version of the Kakeya problem'.

\section{Notions of dimension}

%The Kakeya Conjecture, that the dimension of any Besicovitch set in $\R^d$ is maximal, i.e.\ $d$, comes in (at least) two flavours. The strong version uses the Hausdorff dimension, the weaker the upper box-counting dimension.  We also consider a third notion,  the Assouad dimension.

To define the Hausdorff dimension, we first define the $s$-dimensional Hausdorff (outer) measure of a set $A$ as
$$
{\mathscr H}^s(A)=\lim_{r\to0}\inf\left\{\sum_j|U_j|^s:\ A\subset\bigcup_j U_j,\ |U_j|\le r\right\},
$$
where $|U|$ denotes the diameter of the set $U$. Then
$$
\dim_{\rm H}(A)=\inf\{s:\ {\mathscr H}^s(A)=0\}.
$$
When $s$ is an integer, ${\mathscr H}^s$ is proportional to $s$-dimensional Lebesgue measure in $\mathbb{R}^s$ \cite{Falc,JCR}.  If one uses packings instead of covers, then one obtains the packing dimension $\dim_{\rm P}(A)$; an   important dual to the Hausdorff dimension.  See \cite[Chapter 3]{Falc} for the precise definition, and also (\ref{packingdim}) for a definition in terms of the upper box-counting dimension, which we now define.

The upper box-counting dimension (also referred to as the (upper) Minkowski dimension) is defined for any bounded set $A$, and has a variety of equivalent definitions, the most common being
\begin{equation}\label{boxdef}
\dim_{\rm B}(A)=\limsup_{\epsilon\to0}\frac{\log N(A,\epsilon)}{-\log\epsilon},
\end{equation}
where $N(A,\epsilon)$ denotes the minimum number of balls of radius $\epsilon$ required to cover $A$.  We will use an equivalent definition in Section \ref{open}. Note that it is an immediate consequence of the definition that $\dim_{\rm B}(A)=\dim_{\rm B}(\overline{A})$ (where $\overline A$ denotes the closure of $A$) in contrast to the Hausdorff and packing dimensions, which do not have this property. Indeed, the packing dimension may be expressed in terms of the upper box dimension as
\begin{equation}\label{packingdim}
\dim_{\rm P}(A) = \inf \{ \sup_i \dim_{\rm B}(A_i) : A = \cup_{i=1}^{\infty} A_i \}.
\end{equation}
see \cite[Proposition 3.8]{Falc}. Replacing the $\limsup$ in (\ref{boxdef}) by a $\liminf$ we obtain the lower box-counting dimension:
\begin{equation}\label{Lboxdef}
\dim_{\rm LB}(A)=\liminf_{\epsilon\to0}\frac{\log N(A,\epsilon)}{-\log\epsilon}.
\end{equation}

Finally, the Assouad dimension is similar to the upper box-counting dimension, but looks to minimise a local quantity rather than a global one.  It is defined by
\begin{align*}%\label{assouaddef}
\dim_{\rm A}(A)&=\inf \Big\{s \geq 0 :\ \mbox{there exists }C>0\mbox{ such that }\\
&\qquad\qquad N(B(x,\delta) \cap A, \epsilon )  \leq C(\delta/\epsilon)^s\mbox{ for every }x\in A,\ 0<\epsilon<\delta\Big\}.
\end{align*}
It is well-known that in general
\[
\dim_{\rm H}(A)\le\dim_{\rm LB}(A)\le\dim_{\rm B}(A) \le\dim_{\rm A}(A)
\]
and
\[
\dim_{\rm H}(A)\le\dim_{\rm P}(A)\le\dim_{\rm B}(A)
\]
for any bounded set $A$ \cite{Falc,JCR}. The lower box dimension and the packing dimension are not generally comparable.

\section{The cut-and-move technique}

Our first tool is the following simple `cut-and-move' (decomposition and shifting) lemma, which works for any dimension that is monotonic, stable under finite unions, and translation invariant. A dimension is \emph{monotonic} if
$$
A\subseteq B\qquad\Rightarrow\qquad \dim(A)\le\dim(B),
$$
\emph{stable under finite unions} if
\begin{equation}\label{stable}
\dim\left(\bigcup_{j=1}^n A_j\right)=\max_j\dim(A_j),
\end{equation}
and \emph{translation invariant} if
$$
\dim(T_aA)=\dim(A),\qquad\mbox{where}\qquad T_aA=\{x-a:\ x\in A\}.
$$

These properties are enjoyed by most of the common dimensions, including: Hausdorff, packing, upper-box (Minkowski), and Assouad (see Falconer \cite{Falc} or Robinson \cite{JCR} for a survey of dimensions). Among these, the Hausdorff and packing dimensions are also stable under countable unions (take a countable union in (\ref{stable}), replacing the $\max$ by a $\sup$ on the right-hand side).

We denote by ${\mathscr B}(\R^d)$ the collection of all bounded subsets of $\R^d$.

\begin{lemma}\label{shift}
Suppose that $\dim\colon {\mathscr B}(\R^d)\to[0,\infty)$ is monotonic, stable under finite unions, and translation invariant. Let $K$ be any bounded subset of $\R^d$ with
\begin{equation}\label{Xis}
K=\bigcup_{j=1}^n K_j,
\end{equation}
where the union need not be disjoint. Then for any $\{a_j\}_{j=1}^n\subset \R^d$,
\begin{equation}\label{moved}
\dim(K)=\dim\left(\bigcup_{j=1}^n T_{a_j}K_j\right).
\end{equation}
The same statements remain valid for the lower box-counting dimension (which is not stable under finite unions). If $\dim$ is stable under countable unions then one can allow countable unions in (\ref{Xis}) and (\ref{moved}).
\end{lemma}

\begin{proof}
  Since $K_j\subseteq K$ and $\dim$ is monotonic, $\dim(K_j)\le\dim(K)$. Since $\dim$ is stable under finite unions, $\dim(K)=\max_j \dim(K_j)$. Since the dimension is unaffected by translations, $\dim(T_{a_j}K_j)=\dim(K_j)$. It follows that
  $$
  \dim(K)=\max_j \dim(K_j)=\max_j \dim(T_{a_j}K_j)=\dim\left(\cup_{j=1}^n T_{a_j}K_j\right).
  $$
  
  In the case of the lower box-counting dimension, take $\delta>0$ and let $\mathcal{U}$ be a $\delta$-cover of $\cup_{j=1}^n T_{a_j}K_j.$  Then
\[
\cup_{j=1}^n T_{a_j}^{-1}(\mathcal{U})
\]
is a $\delta$-cover of $K$ and therefore
\[
 \dim_{\mathrm{LB}}(K)  \leq \liminf_{\delta\to0} \left( \frac{\log N(\cup_{j=1}^n T_{a_j}K_j,\delta)}{-\log\delta} + \frac{\log n}{-\log \delta} \right) =\dim_{\mathrm{LB}}\left(\cup_{j=1}^n T_{a_j}K_j\right)
\]
as required.  The opposite inequality is proved similarly by starting with a cover of $K$ and then taking all forward images of covering sets by $T_{a_j}$.
\end{proof}

%The lower box-counting dimension is \emph{not} stable under finite unions, but the conclusion of Lemma \ref{shift} is still valid.
%
%\begin{lemma}\label{shift2}
%Let $K$ be any bounded subset of $\R^d$ with
%\begin{equation}\label{Xis2}
%K=\bigcup_{j=1}^n K_j,
%\end{equation}
%where the union need not be disjoint. Then for any $\{a_j\}_{j=1}^n\subset\R^d$,
%\begin{equation}\label{moved2}
%\dim_{\mathrm{LB}}(K)=\dim_{\mathrm{LB}}\left(\bigcup_{j=1}^n T_{a_j}K_j\right).
%\end{equation}
%\end{lemma}
%
%\begin{proof}
%
%\end{proof}

A  simple application of this procedure yields the following result. It says that we can move
all the line segments in any Kakeya set to within an arbitrarily small distance
of the origin without altering its dimension; see Figure 1.

    \begin{lemma}\label{first}
Suppose that $\dim\colon {\mathscr B}(\R^d)\to[0,\infty)$ is any dimension satisfying the conditions of Lemma \ref{shift}, or is the lower box-counting dimension.   Given a Besicovitch set $K$, for any $\epsilon>0$ there exists another Besicovitch set $\hat K$ such that $\dim(\hat K)=\dim(K)$ and $\hat K$ consists of unit line segments whose centres lie within $\epsilon$ of the origin.
  \end{lemma}

  \begin{proof}
Choose $M>0$ such that $K$ is contained in $(-M,M)^d$, and cover the set $(-M,M)^d$ with a disjoint collection of $d$-dimensional cubes $\prod_{i=1}^d[x_i,x'_i)$ with sides of length $M/n$ for some integer $n$ such that $M/n<\epsilon/\sqrt d$; denote this family of cubes by $\{Q_j\}_{j=1}^N$, and their centres by $a_j$. If $K_j$ is taken to be the set of unit line segments constituting $K$ whose centers lie in $Q_j$ then (\ref{Xis}) holds, and hence using Lemma \ref{shift} the set
$$
\hat K=\bigcup_{j=1}^N T_{a_j}K_j
$$
has the properties required.
\end{proof}

\begin{figure}[H]
\centering
\begin{subfigure}{0.4\textwidth}
\includegraphics[width=\textwidth]{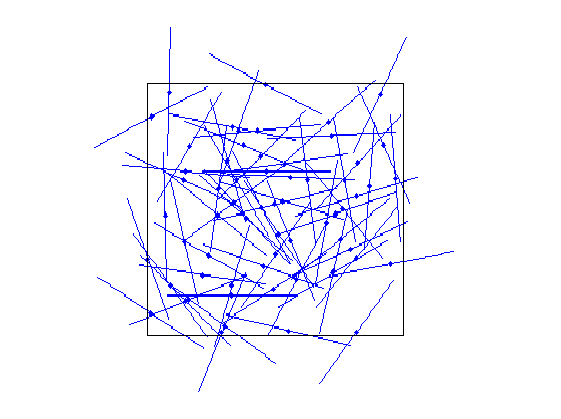}
\end{subfigure}
\begin{subfigure}{0.4\textwidth}
\includegraphics[width=\textwidth]{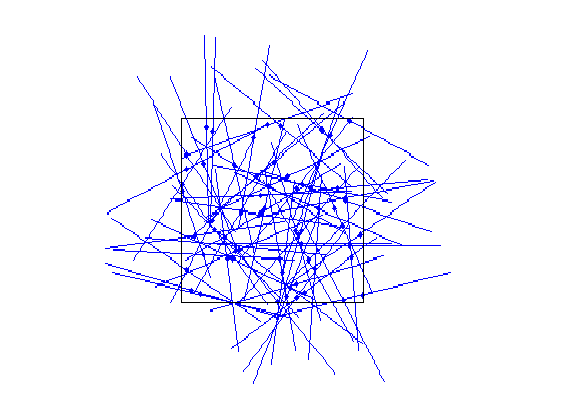}
\end{subfigure}
\begin{subfigure}{0.4\textwidth}
\includegraphics[width=\textwidth]{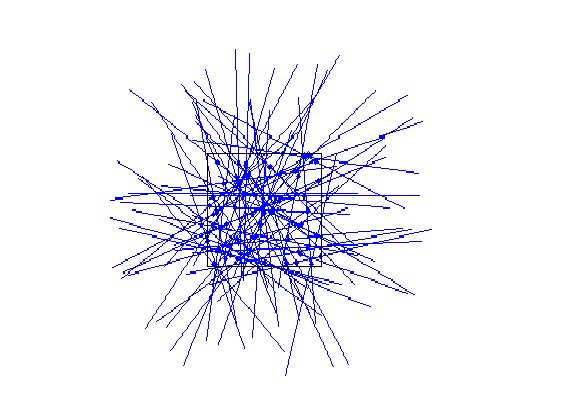}
\end{subfigure}
\begin{subfigure}{0.4\textwidth}
\includegraphics[width=\textwidth]{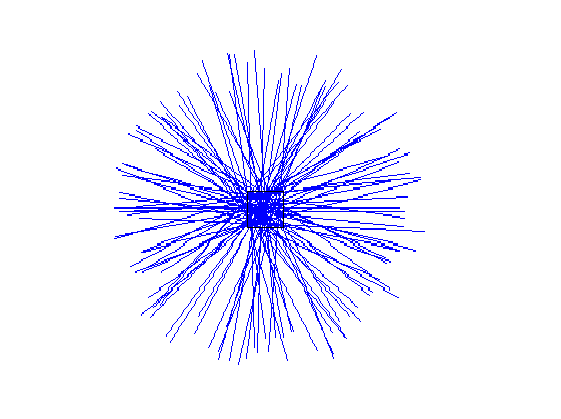}
\end{subfigure}
\caption{The `cut-and-move' procedure of Lemma \ref{shift}, as used in the proof of Lemma \ref{first}, applied to a `discrete' Besicovitch set consisting of 64 line segments (top left) into $4^n$ portions, with $n=1$ (top right), $n=2$ (bottom left), $n=4$ (bottom right). Centres (marked with dots) lie within the squares of sides $4$, $2$, $1$, and $1/4$, respectively.}
\end{figure}

Since the Hausdorff dimension is stable under countable unions, the following improvement holds for Hausdorff dimension.

  \begin{lemma}
    Given a Besicovitch set $K$, there exists a set $\tilde K$ such that $\dim_{\rm H}(\tilde K)=\dim_{\rm H}(K)$ and for every $\epsilon>0$ and every direction there is a unit line segment in $\tilde K$ in the given direction whose centre lies within $\epsilon$ of the origin.
  \end{lemma}

  \begin{proof}
    Let $K_j$ be the set $\hat K$ from Lemma \ref{first} constructed with $\epsilon=2^{-j}$, and set $\tilde K=\cup_jK_j$.
  \end{proof}

Of course, this result does not say anything precise about the Hausdorff dimension of Besicovitch sets, but we find it a useful heuristic that provides a strong indication of the plausibility of the conjecture in a perhaps striking way.

If one could only push this result a little further to yield a set with a line segment in every direction whose centre \emph{was} the origin (i.e.\ the ball of radius 1/2) then this would of course yield a proof of the Kakeya Conjecture. We can achieve this by allowing a small perturbation of our original set, which is the idea behind the density result of Section \ref{density}.

  \section{Non-trivial but easy lower bounds for the  dimension of Besicovitch sets}

In this section we present some simple folklore arguments which give non-trivial lower bounds for the dimensions of Besicovitch sets.   While there are stronger results available, we think that the simplicity of the following argument is appealing in its own right. 

We begin with a very simple result for the lower box dimension.

\begin{theorem}\label{LBH}
If $K$ is a Besivovitch set in $\R^d$ then $\dim_{\rm LB}(K)\ge d/2$.
\end{theorem}

\begin{proof}
  Consider the set $K\times K\subset\R^{2d}$. Using elementary properties of the lower box dimension, $\dim_{\rm LB}(K\times K)\le 2\,\dim_{\rm LB}(K)$. Also, note that the map $\phi:\R^d\times\R^d$ defined by $(x,y)\to x-y$ is Lipschitz from $\R^d\times\R^d$ into $\R^d$, and so $\dim_{\rm LB}[\phi(K\times K)]\le\dim_{\rm LB}(K\times K)$. However, the set $\phi(K\times K)$ contains the unit ball in $\R^d$. Combining these facts yields
  $$
  d\le\dim_{\rm LB}([\phi(K\times K)])\le \dim_{\rm LB}(K\times K)\le 2\dim_{\rm LB}(K),
  $$
  from which the result follows immediately.
\end{proof}

A slightly more sophisticated argument, for which we would like to thank an anonymous referee,  yields an improved result.  We will make use of the following version of Marstrand's slicing theorem, see \cite[Theorem 10.10]{mattila} or (the alluded to higher dimensional analogue of) \cite[Corollary 7.2]{Falc}.

  \begin{theorem}[Marstrand's slicing theorem]\label{slicing}
  Let $F \subseteq \mathbb{R}^n$ and  $E \subseteq V$  be Borel sets where $V$ is a proper subspace of $\mathbb{R}^n$.  If
\[
\dim_{\rm H}  \, F \cap \left(  x+ V^\perp \right) \geq t
\]
for all $x \in E$, then
\[
\dim_{\rm H} F \geq \dim_{\rm H} E + t.
\]
  \end{theorem}

  \begin{theorem}\label{LBD}
    If $K$ is a Besicovitch set in $\R^d$ then
\[
\dim_{\rm LB}(K)\ge (d+1)/2,
\]
\[
\dim_{\rm P}(K) \geq (d+1)/2
\]
and
\[
\dim_{\rm P}(K) + \dim_{\rm H}(K)   \geq d+1.
\]
  \end{theorem}

\begin{proof}
Let $K \subseteq \mathbb{R}^d$ be a Besicovitch set.  Let $v \in \mathbb{R}^d$ and let
\[
A_v = \{ (x,y) \in \mathbb{R}^d \times  \mathbb{R}^d : x-y=v\}.
\]
This gives a $d$-dimensional smooth parametrisation of a family of pairwise disjoint  $d$-dimensional affine subspaces of $\mathbb{R}^{2d}$.  It is easy to verify that $(K \times K) \cap A_v$ contains a line segment for each $v$ such that $|v| < 1$, but we include the details.  Write $v$ in polar coordinates as $v=|v| \theta$ for some $\theta \in S^{d-1}$.  Since $K$ is a Besicovitch set, we know it contains a unit line segment in direction $\theta$, i.e., for some $t \in \mathbb{R}^d$ we have $\{x \theta+t : x \in [0,1] \} \subset K$.  It follows that $\{(x \theta+t, y \theta +t ) : x,y \in [0,1] \} \subseteq K \times K$ and so
\begin{eqnarray*}
(K \times K) \cap A_v &\supset & \{(x \theta+t, y \theta +t ) : x,y \in [0,1] \text{ and } x-y = |v| \} \\
&=& \{((y+|v|)\theta+t, y\theta +t ) : y \in [0,1- |v|]  \} 
\end{eqnarray*}
which is a line segment of length $\sqrt{2}(1- |v|)$ provided $|v| < 1$.    It then follows from Marstrand's slicing theorem (Theorem \ref{slicing} above) that
\begin{equation} \label{keylowerest}
\dim_{\rm H} (K \times K) \geq \dim_{\rm H} \{ v \in \mathbb{R}^d : |v| < 1\} +1 = d+1.
\end{equation}
 The stated results all now follow immediately from classical results on dimensions of products.  In particular, we obtain
\[
d+1 \leq \dim_{\rm H} (K \times K)  \leq \dim_{\rm LB} (K \times K)  \leq 2 \dim_{\rm LB} (K) ,
\]
\[
d+1 \leq \dim_{\rm H} (K \times K)  \leq \dim_{\rm P} (K \times K)  \leq 2 \dim_{\rm P} (K) ,
\]
and
\[
d+1 \leq \dim_{\rm H} (K \times K)  \leq \dim_{\rm H} (K \times K)  \leq \dim_{\rm P}(K) + \dim_{\rm H}(K).
\]
See \cite{Falc, How} for a discussion of the various product formulae we use here. In particular, the inequalities used above apart from (\ref{keylowerest}) hold for \emph{any} set.
\end{proof}
  Unfortunately, we do not have a `self-product formula' for Hausdorff dimension which prevents us from proving the lower bound $(d+1)/2$ for Hausdorff dimension using this approach. This bound does hold, however, but the simplest proof we are aware of is Bourgain's `bush method', see \cite{Bourgain,Green}.

  \section{Density of Besicovitch sets with positive measure}\label{density}

  Let $\mathring S^{d-1}$ denote the unit sphere in $\R^d$ with antipodal points identified. We encode a Besicovitch set in $\R^d$ as a bounded map $f:\mathring S^{d-1}\to\R^d$, where $f(x)$ gives the centre of the unit line segment oriented in the $x$ direction. We denote by $B(\mathring S^{d-1})$ the collection of all such maps, and make this a Banach space by equipping it with the supremum norm.

Given such a map $f$, we define $K(f)$ to be the Besicovitch set encoded by $f$, explicitly
$$
K(f)=\adjustlimits\bigcup_{x\in\mathring S^{d-1}}\bigcup_{t\in[-1/2,1/2]}f(x)+tx.
$$

 We denote by ${\mathscr K}(\R^d)$ the collection of all non-empty compact subsets of $\R^d$.  If we define
$$
\rho(A,B)=\adjustlimits\sup_{a\in A}\inf_{b\in B}|a-b|,
$$
 then the Hausdorff distance (which yields a metric on ${\mathscr K}(\R^d)$) is given by
$$
\dist_{\rm H}(A,B)=\max(\rho(A,B),\rho(B,A)).
$$

We make the following simple observation.

\begin{lemma}\label{continuousK}
If $f,f_0\in B(\mathring S^{d-1})$ then
$$
\dist_{\rm H}\left(\overline{K(f)},\overline{K(f_0)}\right)\le  \| f - f_0\|_\infty.
$$
\end{lemma}

\begin{proof}
  Given $f_0,f\in B(\mathring S^{d-1})$, set $\epsilon=\| f - f_0\|_\infty$ and observe that any point in $K(f_0)$ lies on a unit line segment with centre $f_0(x)$, for some $x\in\mathring S^{d-1}$, and that there is a corresponding point $f(x)\in K(f)$ that lies within $\epsilon$ of $f_0(x)$, since the unit line segment is translated by no more than $\epsilon$ on changing from $f_0$ to $f$. Thus $\rho\left(\overline{K(f)},\overline{K(f_0)}\right)\le\epsilon$. This argument is symmetric between $K(f)$ and $K(f_0)$, and so the desired result follows.
\end{proof}

  We now employ a variant of the arguments of the previous section to prove the density of maps $f$ for which $\dim_{\rm H}(K(f))=d$, and more strongly for which $\mu(K(f))>0$ (where $\mu(A)$ denotes the $d$-dimensional Lebesgue measure of $A$). We note that Besicovitch sets with zero Lebesgue measure were constructed by Besicovitch \cite{B}.

\begin{proposition} \label{denseprop}
  The collection of functions $f\in B(\mathring S^{d-1})$ such that
$$\mu(K(f))>0$$
and, in particular,
  $$\dim_{\rm H}(K(f))=d$$
  is dense.
\end{proposition}

\begin{proof}
Take $f_0\in B(\mathring S^{d-1})$ and $\epsilon>0$. Let $M>0$ be such that $K(f_0)$ is contained in $(-M,M)^d$. Now proceed as in the proof of Lemma \ref{first} to find a disjoint family of cubes $\{Q_j\}_{j=1}^N$ that cover $K(f_0)$, whose sides have length $<\epsilon/\sqrt d$ with centres $a_j$. Now we do not shift the components of the set $K(f_0)$, but rather define an element $f\in B(\mathring S^{d-1})$ by setting
$$
f(x)=a_j\qquad\mbox{if}\quad f_0(x)\in Q_j;
$$
clearly $\|f-f_0\|_\infty<\epsilon$.

Now we can write $K(f)$ as the finite union $K(f)=\bigcup_{j=1}^N U_j$, where $U_j$ consists of all those line segments with centre $a_j$, and observe that
\begin{equation}\label{ball}
\bigcup_{j=1}^N T_{a_j}U_j=B(0,1/2).
\end{equation}
Since the Lebesgue measure is subadditive
$$
\sum_{j=1}^N \mu(T_{a_j}U_j)\ge \mu(B(0,1/2)).
$$
It follows that there exists an $i$ such that $\mu(T_{a_i}U_i)\ge\mu(B(0,1/2))/N$,
and so, since the Lebesgue measure is monotonic and invariant under translations,
$$
\mu(K(f))\ge\mu(U_i)=\mu(T_{a_i}U_i)>0.
$$
Finally, since $\mu$ is comparable to $\mathcal{H}^d$ it follows that $\dim_{\rm H}(K(f))=d$.
\end{proof}

\section{A residual collection of Besicovitch sets with maximal dimension}\label{open}

We now show that a residual collection of Besicovitch sets has maximal upper box-counting dimension.
Recall that a set is \emph{nowhere dense} if its closure has empty interior, and a set is \emph{residual} if its complement is the countable union of nowhere dense sets, see Oxtoby \cite{Oxtoby}.

We make use of another equivalent definition of the upper box-counting dimension (see Falconer \cite{Falc}, for example). Let $N_{\rm disj}(A,\epsilon)$ denote the maximum number of disjoint closed $\epsilon$-balls with centres in $A$. Then it is well known that
\begin{equation} \label{boxdef2}
\dim_{\rm B}(A)=\limsup_{\epsilon\to0}\frac{\log N_{\rm disj}(A,\epsilon)}{-\log\epsilon}.
\end{equation}

\begin{theorem}
  The collection of functions $f\in B(\mathring S^{d-1})$ such that
  $$\dim_{\rm B}(K(f))=d$$
  is a residual subset of $B(\mathring S^{d-1})$.
\end{theorem}

\begin{proof}
Define
$$
{\mathscr F}_{m,n}=\{f\in B(\mathring S^{d-1}):\ \exists\ \delta<1/n\mbox{ such that}\ N_{\rm disj}(\overline{K(f)},\delta)> \delta^{1/m-d}\}.
$$
It follows from (\ref{boxdef2}) that
\begin{align*}
\{f\in B(\mathring S^{d-1}):\ \dim_{\rm B}(K(f))=d\}&=\{f\in B(\mathring S^{d-1}):\ \dim_{\rm B}(\overline{K(f)})=d\}\\
&=\bigcap_{n=1}^\infty\bigcap_{m=1}^\infty{\mathscr F}_{m,n},
\end{align*}
%\begin{align*}
%\bigcap_{i=1}^\infty\bigcap_{j=1}^\infty& \{f\in B(\mathring S^{d-1}):\ N_{\rm disj}(\overline{K(f)},1/i)> i^{d-1/j}\}\\
%&=\{f\in B(\mathring S^{d-1}):\ \dim_{\rm B}(\overline{K(f)})=d\}\\
%&=\{f\in B(\mathring S^{d-1}):\ \dim_{\rm B}(K(f))=d\},
%\end{align*}
since $\dim_{\rm B}(A)=\dim_{\rm B}(\overline{A})$.  Note that we have already shown that for each $m,n\in\N$, the set ${\mathscr F}_{m,n}$ is dense, since it contains all those $f$ for which $\dim_{\rm B}(K(f))=d$, see Proposition \ref{denseprop}.

Recall that ${\mathscr K}(\R^d)$ is the collection of all non-empty compact subsets of $\R^d$, metrised by the Hausdorff distance $\dist_{\rm H}$. It follows from the result of Lemma \ref{continuousK} that the map $f\mapsto\overline {K(f)}$ is continuous from $B(\mathring S^{d-1})$ into ${\mathscr K}(\R^d)$ and, therefore, it remains only to show that for any $\epsilon$ and $r>0$ the set
$$
F_{\epsilon,r}:=\{A\in{\mathscr K}(\R^d):\ \exists\ \delta<\epsilon\ \mbox{such that}\ N_{\rm disj}(A,\delta)>\delta^{-r}\}
$$
is open. Given any $A\in F_{\epsilon,r}$, suppose that there are $N>\delta^{-r}$ disjoint closed balls of radius $\delta<\epsilon$ with centres $\{a_j\}\subset A$. Then these balls are all at least some distance $\eta>0$ apart, and so any set $B$ with $\dist_{\rm H}(A,B)<\eta/2$ contains points $b_j$ with $|a_j-b_j|<\eta/2$ and the $\delta$ balls with centres $\{b_j\}$ are still disjoint. Thus $N_{\rm disj}(B,\delta)\ge N>\delta^{-r}$ and hence $B\in F_{\epsilon,r}$, which is sufficient.

The result now follows since
$$
\{f\in B(\mathring S^{d-1}):\ \dim_{\rm B}(K(f))=d\}=\bigcap_{m,n}{\mathscr F}_{m,n}
$$
is the countable intersection of open dense sets and therefore residual.\end{proof}

An interesting further problem would be to consider whether the Baire generic Besicovitch sets (with respect to our parameterisation) also have full Hausdorff (or lower box) dimension and whether they have zero or positive Lebesgue measure.  Of course, if one believes the Kakeya conjecture itself, then the Baire generic Besicovitch sets also have full Hausdorff dimension, but for the question of measure it is not so clear what to conjecture.
K\"orner \cite{K} proved that Baire generic Besicovitch sets have zero measure, but his parameterisation was different from ours and so the result in our setting does not necessarily follow from \cite{K}.

\section{The Assouad dimension of Besicovitch sets with half-infinite lines}

In this section we prove that any set $K \subseteq \mathbb{R}^d$ which contains a half-infinite line in every direction has full Assouad dimension.  More precisely, let $K \subseteq \mathbb{R}^d$ be such that for all $\theta \in S^{d-1}$ there exists a translation $t_\theta \in \mathbb{R}^d$ such that
\begin{equation}\label{whatstt}
\{ \lambda \theta + t_\theta : \lambda \in [0, \infty) \} \subseteq K.
\end{equation}
We refer to such a set as a `half-extended Besicovitch set'.

\begin{theorem} \label{assouad}
Any half-extended Besicovitch set $K \subseteq \mathbb{R}^d$ has full Assouad dimension, i.e., $\dim_{ \rm A} K = d$.
\end{theorem}

We note that this also implies the same result for `fully-extended Besicovitch sets'; sets containing doubly infinite lines in every direction.

Half-extended and fully-extended Besicovitch sets are relevant in the study of the classical Kakeya problem.  For example, \emph{Keleti's Line Segment Extension Conjecture} \cite{keleti} is that extending any collection of line segments in $\mathbb{R}^d$ to the corresponding collection of doubly infinite lines in the same directions does not alter the Hausdorff dimension.  Keleti proved that this conjecture, if true for a particular $d$, would imply that any (classical) Besicovitch set in $\mathbb{R}^d$ had Hausdorff dimension at least $d-1$ and, if true for all $d$, would imply that  any bounded (classical) Besicovitch set in $\mathbb{R}^d$ had upper box dimension $d$.  Keleti proved that his conjecture is true in the plane.

One of the most common ways of estimating the Assouad dimension from below is to use weak tangents, see \cite{mackaytyson}.  Roughly speaking this says that if you zoom in on a set, any Hausdorff limit you obtain has Assouad dimension no greater than the original set. Here we use this approach, but `zoom out' rather than `zoom in'.

\begin{proposition} \label{tangent}
Let $A \subset \mathbb{R}^d$ and $(S_k)_k$ be a sequence of similarity self-maps on $\mathbb{R}^d$.  Suppose that
\[
S_k(A) \cap \overline{B}_d(0,1) \to \hat A
\]
where the convergence is in the Hausdorff metric and $\overline{B}_d(0,1)$ is the closed unit ball in $\mathbb{R}^d$.  Then  $\dim_{ \rm A} A \geq \dim_{ \rm A} \hat A$.
\end{proposition}

\begin{proof}
Let $s> \dim_{ \rm A} A$ which means there exists $C>0$ such that for all $x \in A$ and all $0<\epsilon'<\delta'$ we have
\[
N(B(x,\delta') \cap A, \epsilon' )  \leq C(\delta'/\epsilon')^s.
\]
Let $c_k>0$ be the (uniform) contraction/expansion rate of $S_k$, i.e. the constant such that
\[
|S_k(x) - S_k(y) | = c_k |x-y |
\]
for all $x,y \in \mathbb{R}^d$.    Let $\hat x \in \hat A$ and fix $0<\epsilon<\delta$. Choose $k$ sufficiently large to ensure that the Hausdorff distance between $S_k(A) \cap \overline{B}_d(0,1) $ and $\hat A$ is strictly less than $\epsilon/2$ and choose $x \in S_k(A) \cap \overline{B}_d(0,1)$ which is strictly closer than $\epsilon/2$ to $\hat x$.  Let $\{U_i\}_i$ be a cover of $B(S_k^{-1}(x),2\delta c_k^{-1}) \cap A$ by fewer than
 \[
 C\left(\frac{2\delta c_k^{-1}}{c_k^{-1} \epsilon/2}\right)^s =    C 4^s (\delta/\epsilon)^s
\]
open balls of radius $c_k^{-1} \epsilon/2$.  It follows that  $\{S_k(U_i)\}_i$ is a cover of $B(x,2\delta ) \cap S_k(A)$ by open balls of radius $\epsilon/2$.  Doubling the radius of each of these balls we obtain a cover of $B(\hat x,\delta ) \cap \hat A$ by open balls of radius $\epsilon$. Therefore
\[
N(B(\hat x,\delta) \cap \hat A, \epsilon )  \leq   C 4^s (\delta/\epsilon)^s
\]
which proves that $ \dim_{ \rm A} \hat A \leq s$ and letting $s \to \dim_{ \rm A} A$ completes the proof.
\end{proof}

Theorem \ref{assouad} now follows immediately from Proposition \ref{tangent} and the fact that the unit ball itself is a tangent to $K$.  Note that, unlike previous work, we zoom out on $K$ to obtain the necessary dimension estimates.

\begin{proposition}
Let $K \subseteq \mathbb{R}^d$ be a half-extended Besicovitch set and let $(S_k)_k$ be defined by $S_k(x) = x/k$.  Then
\[
S_k(K) \cap \overline{B}_d(0,1) \to \overline{B}_d(0,1)
\]
where the convergence is in the Hausdorff metric.
\end{proposition}

\begin{proof}
Let $\Delta_n \subset S^{d-1}$ be a  sequence of finite sets which converge in the Hausdorff metric to $ S^{d-1}$, where $ S^{d-1}$ is metricised with any reasonable metric inherited from the ambient space, such as the spherical metric.  For example one could take $\Delta_n$ to be a maximal $1/n$ separated subset of $S^{d-1}$.  Let
\[
\Delta^*_n = \{ t \theta : \theta \in \Delta_n, \, t \in [0,1]\} \subseteq \overline{B}_d(0,1)
\]
and observe that $\Delta^*_n$ converges to $\overline{B}_d(0,1)$ in the Hausdorff metric.  Let $\epsilon>0$ and choose $n$ sufficiently large to ensure that $\Delta^*_n$ is $\epsilon$-close to $\overline{B}_d(0,1)$ in the Hausdorff metric and also choose $k$ sufficiently large (depending on $n$) to ensure that $\lvert S_k(t_\theta)\rvert <\epsilon / \sqrt{2}$ for all $\theta \in \Delta_n$  (we may do this since $\Delta_n$ is finite and therefore bounded (with a bound that depends on $n$).  Recall that $t_\theta$ is the translation associated to the direction $\theta$ (see (\ref{whatstt})).  Clearly the $\epsilon$-neighbourhood of $S_k(K) \cap \overline{B}_d(0,1)$ contains $\Delta^*_n$ and the $\epsilon$-neighbourhood of $\Delta^*_n$ contains $\overline{B}_d(0,1)$.  This means that the $2\epsilon$-neighbourhood of $S_k(K) \cap \overline{B}_d(0,1)$ contains $\overline{B}_d(0,1)$ and so the Hausdorff distance between $S_k(K) \cap \overline{B}_d(0,1)$ and $\overline{B}_d(0,1)$ is bounded by $2 \epsilon$ completing the proof.
\end{proof}

\section*{Acknowledgments}

JMF was supported by the EPSRC grant EP/J013560/1 when at the University of Warwick and by the Leverhulme Trust Research Fellowship RF-2016-500 when at the University of St Andrews (current).  JCR was supported by the EPSRC Leadership Fellowship EP/G007470/1. This grant also partially supported the sabbatical visit by EJO to Warwick for the academic year 2013/14 when this work began. JCR would like to acknowledge the influence of Miles Caddick's Warwick MMath project \cite{C} (supervised by Jose Rodrigo), which provided a comprehensive introduction to the Kakeya Conjecture and its connections with harmonic analysis.  The authors are grateful to Alex Iosevich for useful discussions and to Jouni Luukkainen for making several helpful comments on the manuscript.  Finally, the authors are very grateful to an anonymous referee for making several  useful comments on the paper including a simplification of the proof of Theorem 4.1 and an improvement of the result, given as Theorem 4.2.

\bibliographystyle{model1-num-names}
\bibliography{<your-bib-database>}

\begin{thebibliography}{00}

\bibitem{B}
 Besicovitch, A. (1928) On Kakeya's problem and a similar one. \emph{Math. Zeits.} {\bf 27}, 312-–320.

\bibitem{Bourgain}
 Bourgain, J. (1991) Besicovitch type maximal operators and applications to Fourier analysis. \emph{Geom. Funct. Anal.} {\bf 1}, 147--187.

 \bibitem{C}
 Caddick, M. (2014) \emph{Multipliers and Bochner--Riesz means}. Warwick MMath project.

\bibitem{D}
Davies, R. (1971) Some remarks on the Kakeya problem. \emph{Proc. Cambridge Philos. Soc.} {\bf 69}, 417--421.

\bibitem{Falc}
Falconer, K. J. (2014) {\em Fractal Geometry: Mathematical Foundations and Applications},
John Wiley, 3rd Ed.

\bibitem{Green}
Green, B. (2013) \textit{Restriction and Kakeya Phenomena}. {\tt http://people.maths.ox.ac.uk/greenbj/papers/rkp.pdf}

\bibitem{How}
Howroyd, J. D. (1996) On Hausdorff and packing dimension of product spaces,
\emph{Math. Proc. Cambridge Philos. Soc.}, {\bf 119},  715--727.

\bibitem{Iosevich}
 Iosevich, A., Morgan,  H. and  Pakianathan, J. (2011)
On directions determined by subsets of vector spaces over finite fields,
\emph{Integers}, {\bf 11}, 9 pp.

\bibitem{KT}
Katz, N.H. \& Tao, T. (2002) New bounds for Kakeya problems. \emph{J. Anal. Math.} {\bf 87}, 231-–263.

\bibitem{KT2}
Katz, N.H. \& Tao, T. (2002) Recent progress on the Kakeya conjecture. \emph{Publ. Mat.}, Proceedings of the 6th International Conference on Harmonic Analysis, 161–179.


\bibitem{XKT} Katz, N.H., {\L}aba, I., \& Tao, T. (2000) An improved bound on the Minkowski dimension of Besicovitch sets in ${\bf R}\sp 3$. \emph{Ann. of Math.} {\bf 152}, 383--446.


\bibitem{keleti}
Keleti, T. (2016)
Are lines much bigger than line segments? \emph{Proc. Amer. Math. Soc.} {\bf  144},  1535--1541.

\bibitem{K} K\"orner, T.W. (2003) Besicovitch via Baire. \emph{Studia Math.} {\bf 158}, 65--78.

%\bibitem{mackay}
%Mackay, J. M. (2011)
%Assouad dimension of self-affine carpets. \emph{Conform. Geom. Dyn.} {\bf 15}, 177--187.

\bibitem{mackaytyson}
Mackay, J. M. and Tyson, J. T. (2010)
\emph{Conformal dimension. Theory and application},
University Lecture Series, 54. American Mathematical Society, Providence, RI.

\bibitem{mattila}
Mattila, P. (1995) \textit{Geometry of sets and measures in Euclidean spaces}, Cambridge Studies in Advanced Mathematics No.~44, Cambridge University Press.

\bibitem{Oxtoby} Oxtoby, J.C. (1980) {\it Measure and category}, second edition,  Springer-Verlag, New York-Berlin.


\bibitem{JCR} Robinson, J.C. (2011) \emph{Dimensions, Embeddings, and Attractors.} Cambridge University Press, Cambridge, UK.

\bibitem{W}
Wolff, T. (1995). An improved bound for Kakeya type maximal functions. \emph{Rev. Mat. Iberoamericana} {\bf 11}, 651-–674.

\end{thebibliography}

\end{document}